\documentclass{amsart}

\newcommand{\Weak}{{\rm Weak-}}

\usepackage{amsthm}
\usepackage{pgf,tikz}
\usepackage{amsmath}
\usepackage{amsfonts}
\usepackage{amssymb}
\usepackage{graphicx}%

\usetikzlibrary{arrows}
\usepackage{mathptmx}
\usepackage{helvet}
\usepackage{courier}
\usepackage{type1cm}

\usepackage{multicol}
\usepackage[bottom]{footmisc}

\newtheorem{theorem}{Theorem}
\newtheorem{corollary}[theorem]{Corollary}

\newtheorem{lemma}[theorem]{Lemma}

\providecommand{\U}[1]{\protect\rule{.1in}{.1in}}
\usetikzlibrary{arrows}
\allowdisplaybreaks

\begin{document}

\title[Integer points in translated domains]{$L^{p}$ and $\Weak L^{p}$ estimates for the number of integer
points in translated domains}

\author[L. Brandolini]{Luca Brandolini}

\address{Dipartimento di Ingegneria Gestionale, dell'Informazione e della Produzione,
                Universit\`{a} di Bergamo,
                Viale Marconi 5,
                 24044 Dalmine (BG), Italy.}
\email{luca.brandolini@unibg.it}  
             
            \author[L. Colzani]{Leonardo Colzani}
            \address{Dipartimento di Matematica e Applicazioni,
               Edificio U5,
               Universit\`{a} di Milano Bicocca,
               Via R.Cozzi 53, 20125 Milano, Italy.}
\email               {leonardo.colzani@unimib.it}
          
                \author[G. Gigante]{Giacomo Gigante} 
\address{                Dipartimento di Ingegneria Gestionale, dell'Informazione e della Produzione,
                Universit\`{a} di Bergamo,
                Viale Marconi 5,
                 24044 Dalmine (BG), Italy.}              
                 \email{giacomo.gigante@unibg.it}
                
                \author[G. Travaglini]{Giancarlo Travaglini}
\address {Dipartimento di Statistica e Metodi Quantitativi,
                 Edificio U7,
                 Universit\`{a} di Milano-Bicocca,
                 Via Bicocca degli Arcimboldi 8,
                 20126 Milano, Italy.}
\email      {giancarlo.travaglini@unimib.it}

\begin{abstract}
Revisiting and extending a recent result of M.Huxley, we estimate the
$L^{p}\left(  \mathbb{T}^{d}\right)  $ and $\Weak L^{p}\left(
\mathbb{T}^{d}\right)  $ norms of the discrepancy between the volume
and the number of integer points in translated domains.

\end{abstract}


\maketitle

In this paper we estimate different norms of the discrepancy between the
volume and the number of integer points in dilated and translated copies
$R\Omega-x$ of a bounded convex domain $\Omega\subset\mathbb{R}^{d}$ having
positive measure. The above number of integer points is a periodic function of
the translation variable $x$, with Fourier expansion
\begin{align*}
\sum_{k\in\mathbb{Z}^{d}}\chi_{R\Omega-x}(k)  &  =\sum_{n\in\mathbb{Z}^{d}%
}\left(  \int_{\mathbb{T}^{d}}\sum_{k\in\mathbb{Z}^{d}}\chi_{R\Omega-y}%
(k)\exp(-2\pi iny)dy\right)  \exp(2\pi inx)\\
&  =\sum_{n\in\mathbb{Z}^{d}}\left(  \int_{\mathbb{R}^{d}}\chi_{R\Omega
}(y)\exp(-2\pi iny)dy\right)  \exp(2\pi inx)\\
&  =\sum_{n\in\mathbb{Z}^{d}}R^{d}\widehat{\chi_{\Omega}}(Rn)\exp(2\pi inx).
\end{align*}
The last equality is in the $L^{2}$ sense.
It follows that the discrepancy function
\[
\mathcal{D}\left(  R\Omega-x\right)  =\sum_{k\in\mathbb{Z}^{d}}\chi
_{R\Omega-x}(k)-R^{d}\left\vert \Omega\right\vert
\]
has Fourier expansion%
\[
\sum_{n\in \mathbb{Z}^{d}\setminus \{0\} }R^{d}\widehat{\chi_{\Omega}}(Rn)\exp(2\pi
inx).
\]

If $\Omega$ is a bounded convex domain in $\mathbb{R}^{d}$ with smooth
boundary having positive Gaussian curvature then%
\[
\left\vert \widehat{\chi_{\Omega}}(\xi)\right\vert \leqslant C\left\vert
\xi\right\vert ^{-\left(  d+1\right)  /2}.
\]
See \cite[Chapter 8]{S}. Kendall \cite{K} observed that the Fourier expansion of the discrepancy and
the above estimate for the Fourier transform of a convex domain give %
\[
\left\{  \int_{\mathbb{T}^{d}}\left\vert \mathcal{D}\left(  R\Omega-x\right)
\right\vert ^{2}dx\right\}  ^{1/2}\leqslant CR^{\left(  d-1\right)  /2}.
\]

Using a smoothing argument and the Poisson summation formula, Herz \cite{H}
and Hlawka \cite{Hl} (see also \cite{T}) proved that%
\[
\sup_{x\in\mathbb{T}^{d}}\left\{  \left\vert \mathcal{D}\left(  R\Omega
-x\right)  \right\vert \right\}  \leqslant CR^{d(d-1)/(d+1)}.
\]

Interpolating the above two upper bounds between $L^{2}$ and $L^{\infty}$
gives a poor estimate. Indeed when $d=2$ interpolation gives%
\[
\left\{  \int_{\mathbb{T}^{2}}\left\vert \mathcal{D}\left(  R\Omega-x\right)
\right\vert ^{p}dx\right\}  ^{1/p}\leqslant CR^{\left(  2p-1\right)  /\left(
3p\right)  },
\]
while M.Huxley \cite{Hu} has recently showed a more
interesting estimate: If $\Omega$ is a planar convex body having
boundary with continuous and positive curvature then%
\[
\left\{  \int_{\mathbb{T}^{2}}\left\vert \mathcal{D}\left(  R\Omega-x\right)
\right\vert ^{4}dx\right\}  ^{1/4}\leqslant CR^{1/2}\log^{1/4}\left(
R\right)  .
\]
That is, the upper estimate for the $L^{2}$ discrepancy extends, up to a logarithm, to $L^{4}$.

Huxley's proof seems to be tailored for the planar case and for the exponent
$p=4$, where one can apply Parseval equality to the square of the discrepancy
function. Huxley also asked for an analog of his result for $d>2$. 
Here we will give a possible answer and our approach will be to obtain $L^p$ results through $\Weak L^p$ techniques.

We recall that the spaces $L^{p}\left(  \mathbb{X},\mu\right)  $ and
$\Weak L^{p}\left(  \mathbb{X},\mu\right)  $, $0<p<+\infty$, are defined by the
quasi norms
\begin{gather*}
\left\Vert f\right\Vert _{L^{p}\left(  \mathbb{X},\mu\right)  }=\left\{
\int_{\mathbb{X}}\left\vert f\left(  x\right)  \right\vert ^{p}d\mu\left(
x\right)  \right\}  ^{1/p},\\
\left\Vert f\right\Vert _{\Weak L^{p}\left(  \mathbb{X},\mu\right)  }%
=\sup_{t>0}\left\{  t^{p}\mu\left\{  x\in\mathbb{X},\ \left\vert f\left(
x\right)  \right\vert >t\right\}  \right\}  ^{1/p}.
\end{gather*}

The space $\Weak L^{p}\left(  \mathbb{X},\mu\right)  $ is the case $q=+\infty$
of the Lorentz spaces $L^{p,q}\left(  \mathbb{X},\mu\right)  $ (see e.g.
\cite[Chapter 1, \S 3]{B-L} or \cite[Chapter 5, \S 3]{SW}). Finally, the space
$L^{\infty}\left(  \mathbb{X},\mu\right)  $ is defined by the norm
\[
\left\Vert f\right\Vert _{L^{\infty}\left(  \mathbb{X},\mu\right)  }%
=\inf\left\{  t>0:\ \mu\left\{  x\in\mathbb{X}:\ \left\vert f\left(  x\right)
\right\vert >t\right\}  =0\right\}  .
\]

In what follows $\left(  \mathbb{X},\mu\right)  $ will be the torus
$\mathbb{T}^{d}$ or the integers $\mathbb{Z}^{d}$ with the respective
translation invariant measures.

If $\mathbb{X}$ has finite measure and $p<s$, then both $L^{p}\left(
\mathbb{X},\mu\right)  $ and $L^{s}\left(  \mathbb{X},\mu\right)  $ are
intermediate between $L^{\infty}\left(  \mathbb{X},\mu\right)  $ and
$\Weak L^{p}\left(  \mathbb{X},\mu\right)  $:
\[
L^{\infty}\left(  \mathbb{X},\mu\right)  \subseteq L^{s}\left(  \mathbb{X}%
,\mu\right)  \subseteq L^{p}\left(  \mathbb{X},\mu\right)  \subseteq
\Weak L^{p}\left(  \mathbb{X},\mu\right)  .
\]
The following is a quantitative counterpart of these inclusions.

\begin{lemma}
\label{Key Lemma}(1) \textit{If }$\mu\left(  \mathbb{X}\right)  =1$\textit{,
then}%
\[
\left\Vert f\right\Vert _{L^{p}\left(  \mathbb{X},\mu\right)  }^{p}%
\leqslant1+p\left\Vert f\right\Vert _{\Weak L^{p}\left(  \mathbb{X},\mu\right)
}^{p}\log_{+}\left(  \left\Vert f\right\Vert _{L^{\infty}\left(
\mathbb{X},\mu\right)  }\right)  .
\]

(2) If $p<s<+\infty$, then%
\[
\left\Vert f\right\Vert _{L^{s}\left(  \mathbb{X},\mu\right)  }^{s}%
\leqslant\frac{s}{s-p}\left\Vert f\right\Vert _{\Weak L^{p}\left(
\mathbb{X},\mu\right)  }^{p}\left\Vert f\right\Vert _{L^{\infty}\left(
\mathbb{X},\mu\right)  }^{s-p}.
\]

\end{lemma}

\begin{proof}
(1) If $\left\Vert f\right\Vert _{L^{\infty}\left(  \mathbb{X},\mu\right)
}\leqslant1$, then also $\left\Vert f\right\Vert _{L^{p}\left(  \mathbb{X}%
,\mu\right)  }\leqslant1$ and the lemma follows. Otherwise,
\begin{align*}
\mu\left\{  x\in\mathbb{X},\ \left\vert f\left(  x\right)  \right\vert
>t\right\}   &  \leqslant\mu\left\{  \mathbb{X}\right\}  =1,\\
\mu\left\{  x\in\mathbb{X},\ \left\vert f\left(  x\right)  \right\vert
>t\right\}   &  \leqslant\left\Vert f\right\Vert _{\Weak L^{p}\left(
\mathbb{X},\mu\right)  }^{p}t^{-p},\\
\mu\left\{  x\in\mathbb{X},\ \left\vert f\left(  x\right)  \right\vert
>t\right\}   &  =0\quad\text{if }t\geqslant\left\Vert f\right\Vert
_{L^{\infty}\left(  \mathbb{X},\mu\right)  }.
\end{align*}
Hence,
\begin{align*}
\left\Vert f\right\Vert _{L^{p}\left(  \mathbb{X},\mu\right)  }^{p}  &
=\int_{0}^{+\infty}pt^{p-1}\mu\left\{  x\in\mathbb{X},\ \left\vert f\left(
x\right)  \right\vert >t\right\}  dt\\
&  \leqslant\int_{0}^{1}pt^{p-1}dt+p\left\Vert f\right\Vert _{\Weak L^{p}%
\left(  \mathbb{X},\mu\right)  }^{p}\int_{1}^{\left\Vert f\right\Vert
_{L^{\infty}\left(  \mathbb{X},\mu\right)  }}\frac{dt}{t}\\
&  =1+p\left\Vert f\right\Vert _{\Weak L^{p}\left(  \mathbb{X},\mu\right)
}^{p}\log\left(  \left\Vert f\right\Vert _{L^{\infty}\left(  \mathbb{X}%
,\mu\right)  }\right)  .
\end{align*}

(2) As before,%
\begin{align*}
\left\Vert f\right\Vert _{L^{s}\left(  \mathbb{X},\mu\right)  }^{s} &
=\int_{0}^{+\infty}st^{s-1}\mu\left\{  x\in\mathbb{X},\ \left\vert f\left(
x\right)  \right\vert >t\right\}  dt\\
&  \leqslant s\left\Vert f\right\Vert _{\Weak L^{p}\left(  \mathbb{X}%
,\mu\right)  }^{p}\int_{0}^{\left\Vert f\right\Vert _{L^{\infty}\left(
\mathbb{X},\mu\right)  }}t^{s-1-p}dt\\
&  =\frac{s}{s-p}\left\Vert f\right\Vert _{\Weak L^{p}\left(  \mathbb{X}%
,\mu\right)  }^{p}\left\Vert f\right\Vert _{L^{\infty}\left(  \mathbb{X}%
,\mu\right)  }^{s-p}.
\end{align*}
\end{proof}
Our first result is a simple application of the Hausdorff-Young inequality.
\begin{theorem}
\label{Thm Estimate}Let $\Omega$\textit{\ be a bounded open set in
}$\mathbb{R}^{d}$.

\textbf{(1)}\textit{\ If }$2\leqslant p<+\infty$\textit{\ and }$1/p+1/q=1$%
\textit{, then}%
\[
\left\Vert \mathcal{D}\left(  R\Omega-x\right)  \right\Vert _{L^{p}\left(
\mathbb{T}^{d}\right)  }\leqslant R^{d}\left\Vert \left\{  \widehat{\chi
_{\Omega}}(Rn)\right\}  _{n\neq0}\right\Vert _{L^{q}\left(  \mathbb{Z}%
^{d}\right)  }.
\]

\end{theorem}

\textbf{(2)}\textit{\ If }$2<p<+\infty$\textit{\ and }$1/p+1/q=1$\textit{,
then}%
\[
\left\Vert \mathcal{D}\left(  R\Omega-x\right)  \right\Vert _{\Weak L^{p}%
\left(  \mathbb{T}^{d}\right)  }\leqslant CR^{d}\left\Vert \left\{
\widehat{\chi_{\Omega}}(Rn)\right\}  _{n\neq0}\right\Vert _{\Weak L^{q}\left(
\mathbb{Z}^{d}\right)  }.
\]

\textbf{(3)}\textit{\ If }$2<p<+\infty$\textit{\ and }$1/p+1/q=1$\textit{,
then}%
\[
\left\Vert \mathcal{D}\left(  R\Omega-x\right)  \right\Vert _{L^{p}\left(
\mathbb{T}^{d}\right)  }\leqslant CR^{d}\log^{1/p}\left(  2+R\right)
\left\Vert \left\{  \widehat{\chi_{\Omega}}(Rn)\right\}  _{n\neq0}\right\Vert
_{\Weak L^{q}\left(  \mathbb{Z}^{d}\right)  }.
\]

\textbf{(4)}\textit{\ If }$1\leqslant p\leqslant+\infty$\textit{, then}%
\[
\left\Vert \mathcal{D}\left(  R\Omega-x\right)  \right\Vert _{L^{p}\left(
\mathbb{T}^{d}\right)  }\geqslant\sup_{n\neq0}\left\{  \left\vert
R^{d}\widehat{\chi_{\Omega}}(Rn)\right\vert \right\}  .
\]

\begin{proof}
Point (1) readily follows from the Fourier expansion of the discrepancy and
the Hausdorff-Young inequality: If $2\leqslant p\leqslant+\infty$\ and
$1/p+1/q=1$, then
\[
\left\Vert f\right\Vert _{L^{p}\left(  \mathbb{T}^{d}\right)  }\leqslant
\left\Vert \widehat{f}\right\Vert _{L^{q}\left(  \mathbb{Z}^{d}\right)  }.
\]
The case $\left(  p,q\right)  =\left(  2,2\right)  $ is 
Parseval's identity. The case $\left(  p,q\right)  =\left(  +\infty,1\right)  $ is
immediate. The intermediate cases follow by the Riesz-Thorin interpolation
theorem. See \cite[Theorem 1.1.1]{B-L} or \cite[Chapter V, \S 1]{SW}.
Similarly, point (2) follows from the Hausdorff-Young inequality for Lorentz
spaces: If $2<p<+\infty$\ and if $1/p+1/q=1$, then%
\[
\left\Vert f\right\Vert _{\Weak L^{p}\left(  \mathbb{T}^{d}\right)  }\leqslant
C\left\Vert \widehat{f}\right\Vert _{\Weak L^{q}\left(  \mathbb{Z}^{d}\right)
}.
\]
The proof of this inequality is by real interpolation between the extreme
cases $L^{2}\rightarrow L^{2}$ and $L^{1}\rightarrow L^{\infty}$. See the
general Marcinkiewicz interpolation theorem \cite[Theorem 5.3.2]{B-L} or
\cite[Chapter V, \S 3]{SW}. Point (3) follows from point (2), Lemma
\ref{Key Lemma}, and the trivial estimate $\left\vert \mathcal{D}\left(
R\Omega-x\right)  \right\vert \leqslant CR^{d}$. Finally, a Fourier
coefficient is dominated by the norm of the function, and point (4) follows.
\end{proof}

The above theorem is quite abstract. In order to obtain explicit results, one
has to estimate the norms of the sequences $\left\{  \widehat{\chi_{\Omega}%
}(Rn)\right\}  _{n\neq0}$. The interest of case (3) is when the $L^{q}\left(
\mathbb{Z}^{d}\right)  $ norm is infinite and the $\Weak L^{q}\left(
\mathbb{Z}^{d}\right)  $ norm is finite.

In order to introduce the next result, we recall that the modulus of
continuity of a characteristic function shows that such a function does not
belong to a Sobolev class $W^{\alpha,2}\left(  \mathbb{R}^{d}\right)  $
whenever $\alpha\geqslant1/2$. See \cite[Chapter 5, \S 5]{SV}. Moreover, in  \cite[Corollary 2.2]{KW} it is proved that for
every set $\Omega$ with finite positive measure, without any regularity
assumption, there exists a constant $C$ such that%
\[
\int_{\left\vert \xi\right\vert >R}\left\vert \widehat{\chi_{\Omega}}\left(
\xi\right)  \right\vert ^{2}d\xi\geqslant CR^{-1}.
\]
It follows that a uniform inequality of the kind
$\left\vert \widehat{\chi_{\Omega}}(\xi)\right\vert \leqslant C\left\vert
\xi\right\vert ^{-\beta}$ cannot hold with $\beta>\left(  d+1\right)  /2$. On
the other hand, this estimate holds with $\beta=\left(  d+1\right)  /2$ if
$\Omega$ is a bounded convex domain with smooth boundary with non-vanishing
Gaussian curvature. See \cite{S}. See also \cite{B-N-W} for possible generalizations 
to convex bodies with smooth boundary containing isolated points  with vanishing
Gaussian curvature.

\begin{corollary}
\label{Corollary3}\textit{Assume that }$\Omega$\textit{\ is a bounded convex
domain such that }%
\[
\left\vert \widehat{\chi_{\Omega}}(\xi)\right\vert \leqslant C\left\vert
\xi\right\vert ^{-\left(  d+1\right)  /2}.
\]
\textbf{(1)}\textit{\ If }$p<2d/\left(  d-1\right)  $\textit{\ and }%
$R>2$\textit{, then }%
\[
\left\Vert \mathcal{D}\left(  R\Omega-x\right)  \right\Vert _{L^{p}\left(
\mathbb{T}^{d}\right)  }\leqslant CR^{\left(  d-1\right)  /2}.
\]
\textbf{(2)}\textit{\ If }$p\leqslant2d/\left(  d-1\right)  $\textit{\ and
}$R>2$\textit{, then }%
\[
\left\Vert \mathcal{D}\left(  R\Omega-x\right)  \right\Vert _{\Weak L^{p}%
\left(  \mathbb{T}^{d}\right)  }\leqslant CR^{\left(  d-1\right)  /2}.
\]
\textbf{(3)}\textit{\ If }$p=2d/\left(  d-1\right)  $\textit{\ and }%
$R>2$\textit{, then }%
\[
\left\Vert \mathcal{D}\left(  R\Omega-x\right)  \right\Vert _{L^{p}\left(
\mathbb{T}^{d}\right)  }\leqslant CR^{\left(  d-1\right)  /2}\log^{\left(d-1\right)/\left(2d\right)}\left(
R\right)  .
\]
\textbf{(4)}\textit{\ If }$p>2d/\left(  d-1\right)  $\textit{\ and }%
$R>2$\textit{, then }%
\[
\left\Vert \mathcal{D}\left(  R\Omega-x\right)  \right\Vert _{L^{p}\left(
\mathbb{T}^{d}\right)  }\leqslant CR^{d\left(  pd-p-d+1\right)  /p\left(
d+1\right)  }.
\]

\end{corollary}

\begin{proof}
Points (1), (2) and (3) follow from Theorem \ref{Thm Estimate}, and the
observation that the sequence $\left\{  \left\vert n\right\vert ^{-\alpha
}\right\}  _{n\neq0}$ is in $L^{q}\left(  \mathbb{Z}^{d}\right)  $ if and only
if $q\alpha>d$, and it is in $\Weak L^{q}\left(  \mathbb{Z}^{d}\right)  $ if
and only if $q\alpha\geqslant d$. Point (4) follows from point (2) with
$p=2d/(d-1)$, the pointwise estimate $\left\vert \mathcal{D}\left(
R\Omega-x\right)  \right\vert \leqslant CR^{d(d-1)/(d+1)}$ proved in \cite{H}
and \cite{Hl}, and (2) in Lemma \ref{Key Lemma}.
\end{proof}

\bigskip

The estimates in the above Corollary for $p<2d/\left(  d-1\right)  $ are
essentially sharp. In order to show this, we first recall the following result
on the Fourier transform of the characteristic function of a convex set.

\begin{theorem}
\label{asymptotic}Let $\Omega\subset\mathbb{R}^{d}$ be a convex body with
smooth boundary having everywhere positive Gaussian curvature. For every
$\xi\in\mathbb{R}^{d}\setminus\left\{  0\right\}  $ let $\sigma\left(
\xi\right)  $ be the unique point on the boundary $\partial\Omega$ with
outward unit normal $\xi/\left\vert \xi\right\vert $. Also let $K\left(
\sigma\left(  \xi\right)  \right)  $ be the Gaussian curvature of
$\partial\Omega$ at $\sigma\left(  \xi\right)  $. Then, as $\left\vert
\xi\right\vert \rightarrow+\infty$, the Fourier transform of $\chi_{\Omega
}\left(  x\right)  $ has the asymptotic expansion%
\begin{align*}
&  \widehat{\chi_{\Omega}}\left(  \xi\right) \\
&  =-\frac{1}{2\pi i}\left\vert \xi\right\vert ^{-\frac{d+1}{2}}\left[
K^{-\frac{1}{2}}\left(  \sigma\left(  \xi\right)  \right)  e^{-2\pi i\left(
\xi\cdot\sigma\left(  \xi\right)  -\frac{d-1}{8}\right)  }-K^{-\frac{1}{2}%
}\left(  \sigma\left(  -\xi\right)  \right)  e^{-2\pi i\left(  \xi\cdot
\sigma\left(  -\xi\right)  +\frac{d-1}{8}\right)  }\right] \\
&  +O\left(  \left\vert \xi\right\vert ^{-\frac{d+3}{2}}\right)  .
\end{align*}

\end{theorem}

\begin{proof}
See e.g. \cite{G-G-V}, \cite{H}, or \cite{Hl}.
\end{proof}

The following result partially complements Corollary \ref{Corollary3}.

\begin{theorem}
\label{ThmDaSotto}Let $\Omega\subset\mathbb{R}^{d}$ be a convex body with
smooth boundary having everywhere positive Gaussian curvature.\newline%
\textbf{(1)} If $\Omega$ is not symmetric about a point, or if the dimension
$d\not \equiv 1\left(  \operatorname{mod}4\right)  $, then for every
$p\geqslant1$ there exists $C>0$ such that for every $R>2$%
\[
\left\Vert \mathcal{D}\left(  R\Omega-x\right)  \right\Vert _{L^{p}\left(
\mathbb{T}^{d}\right)  }\geqslant CR^{\frac{d-1}{2}}.
\]
\textbf{(2)} If $\Omega$ is symmetric about a point and if $d\equiv1\,\left(
\operatorname{mod}4\right)  $ then%
\[
\limsup_{R\rightarrow+\infty}\left\{  R^{-\frac{d-1}{2}}\left\Vert
\mathcal{D}\left(  R\Omega-x\right)  \right\Vert _{L^{p}\left(  \mathbb{T}%
^{d}\right)  }\right\}  >0\ \ \ \ \ \ \ \ \ \ \text{if }p\geqslant1,
\]%
\[
\liminf_{R\rightarrow+\infty}\left\{  R^{-\frac{d-1}{2}}\left\Vert
\mathcal{D}\left(  R\Omega-x\right)  \right\Vert _{L^{p}\left(  \mathbb{T}%
^{d}\right)  }\right\}  =0\ \ \ \ \ \ \text{if }p<\frac{2d}{d-1}.
\]
More precisely, if $p<2d/\left(  d-1\right)  $ there exist $C>0$, and a
sequence $R_{j}\rightarrow+\infty$, such that%
\[
\left\Vert \mathcal{D}\left(  R_{j}\Omega-x\right)  \right\Vert _{L^{p}%
(\mathbb{T}^{d})}\leqslant CR_{j}^{\frac{d-1}{2}}\left(  \frac{\log\left(
R_{j}\right)  }{\log\left(  \log\left(  R_{j}\right)  \right)  }\right)
^{\frac{d-1}{2d}-\frac{1}{p}}.
\]

\end{theorem}

\begin{proof}
In order to prove point (1) observe that,  by Theorem \ref{Thm Estimate},%
\[
\left\Vert \mathcal{D}\left(  R\Omega-x\right)  \right\Vert _{L^{p}\left(
\mathbb{T}^{d}\right)  }\geqslant\sup_{n\neq0}\left\{  \left\vert
R^{d}\widehat{\chi_{\Omega}}(Rn)\right\vert \right\}  .
\]
Moreover, by Theorem \ref{asymptotic}, for $n\neq0$,%
\begin{align*}
&  R^{d}\widehat{\chi_{\Omega}}(Rn)=\\
&  \frac{-1}{2\pi i}R^{\frac{d-1}{2}}\left\vert n\right\vert ^{-\frac{d+1}%
{2}}
 \left[  K^{-\frac{1}{2}}\left(  \sigma\left(  n\right)  \right)
e^{-2\pi i\left(  Rn\cdot\sigma\left(  n\right)  -\frac{d-1}{8}\right)
}-K^{-\frac{1}{2}}\left(  \sigma\left(  -n\right)  \right)  e^{-2\pi i\left(
Rn\cdot\sigma\left(  -n\right)  +\frac{d-1}{8}\right)  }\right] \\
&  ~~~~~+O\left(  R^{\frac{d-3}{2}}\left\vert n\right\vert ^{-\frac{d+3}{2}%
}\right)  .
\end{align*}
If $\Omega$ is not symmetric, it follows that $K\left(  \sigma\left(
u\right)  \right)  $ cannot be symmetric (see \cite[\S 14, p. 133]{B-F}).
Since the set $\left\{  \frac{n}{\left\vert n\right\vert }:n\in\mathbb{Z}%
^{d}\setminus\left\{  0\right\}  \right\}  $ is dense in the unit sphere, by
continuity there exists $m\in\mathbb{Z}^{d}$ such that $
K\left(  \sigma\left(  m\right)  \right)  \neq K\left(  \sigma\left(
-m\right)  \right)  .
$
Then, for this $m$ and  $R$ large enough,%
\begin{align*}
&  \left\vert R^{d}\widehat{\chi_{\Omega}}(Rm)\right\vert \\
&  \geqslant\frac{1}{2\pi}R^{\frac{d-1}{2}}\left\vert m\right\vert
^{-\frac{d+1}{2}}\left\vert K^{-\frac{1}{2}}\left(  \sigma\left(  m\right)
\right)  -K^{-\frac{1}{2}}\left(  \sigma\left(  -m\right)  \right)
\right\vert +O\left(  R^{\frac{d-3}{2}}\left\vert m\right\vert ^{-\frac
{d+3}{2}}\right) \\
&  \geqslant CR^{\frac{d-1}{2}}.
\end{align*}
Assume now that $\Omega$ is symmetric, and translate the center of symmetry to the origin, so that for every $\xi$ we
have $\sigma\left(  -\xi\right)  =-\sigma\left(  \xi\right)  $ and $K\left(
\sigma\left(  \xi\right)  \right)  =K\left(  \sigma\left(  -\xi\right)
\right)  $. Choose $n\neq0$ and observe that%
\[
n\cdot\left(  \sigma\left(  n\right)  -\sigma\left(  -n\right)  \right)
=2n\cdot\sigma\left(  n\right)  \neq0.
\]
Indeed, $n\cdot\sigma\left(  n\right)  =0$ would imply that the center belongs
to the hyperplane tangent to $\partial\Omega$ at $\sigma\left(  n\right)  $,
hence $\Omega$ should have measure $0$. We have%
\begin{align*}
&  \left\vert R^{d}\widehat{\chi_{\Omega}}(Rn)\right\vert \\
&  \geqslant CR^{\frac{d-1}{2}}\left\vert n\right\vert ^{-\frac{d+1}{2}%
}K^{-\frac{1}{2}}\left(  \sigma\left(  n\right)  \right)  \left\vert e^{2\pi
i\left(  2Rn\cdot\sigma\left(  n\right)  -\frac{d-1}{4}\right)  }-1\right\vert
+O\left(  R^{\frac{d-3}{2}}\left\vert n\right\vert ^{-\frac{d+3}{2}}\right)  .
\end{align*}
Let $\left\Vert x\right\Vert $ denote the distance of a real number $x$ from
the integers. If%
\[
\left\Vert 2Rn\cdot\sigma\left(  n\right)  -\frac{d-1}{4}\right\Vert >\frac
{1}{10},
\]
then $\left\vert e^{2\pi i\left(  2Rn\cdot\sigma\left(  n\right)  -\frac
{d-1}{4}\right)  }-1\right\vert >c$ and we have%
\[
\left\vert R^{d}\widehat{\chi_{\Omega}}(Rn)\right\vert \geqslant
cR^{\frac{d-1}{2}}.
\]
Assume now that%
\[
\left\Vert 2Rn\cdot\sigma\left(  n\right)  -\frac{d-1}{4}\right\Vert
\leqslant\frac{1}{10}.
\]
Then%
\[
\left\Vert 4Rn\cdot\sigma\left(  n\right)  -\frac{d-1}{2}\right\Vert
\leqslant\frac{1}{5}.
\]
Since $d\not \equiv 1$ $\left(  \operatorname{mod}4\right)  $, we have%
\[
\left\Vert 4Rn\cdot\sigma\left(  n\right)  -\frac{d-1}{4}\right\Vert
\geqslant\frac{1}{20}.
\]
Applying the previous argument with $2n$ in place of $n$ provides the estimate%
\[
\left\vert R^{d}\widehat{\chi_{\Omega}}(R2n)\right\vert \geqslant
cR^{\frac{d-1}{2}}.
\]

In order to prove point (2), assume that $\Omega$ is symmetric and $d\equiv1$ $\left(
\operatorname{mod}4\right)  $. From the asymptotic estimate of $\widehat{\chi
}_{\Omega}\left(  \xi\right)  $ we obtain%
\begin{align*}
  \left\vert R^{d}\widehat{\chi_{\Omega}}(Rn)\right\vert 
  =\frac{1}{\pi}R^{\frac{d-1}{2}}\left\vert n\right\vert ^{-\frac{d+1}{2}%
}K^{-\frac{1}{2}}\left(  \sigma\left(  n\right)  \right)  \left\vert
\sin\left(  2\pi Rn\cdot\sigma\left(  n\right)  \right)  \right\vert +O\left(
R^{\frac{d-3}{2}}\left\vert n\right\vert ^{-\frac{d+3}{2}}\right)  .
\end{align*}
Since $n\cdot\sigma\left(  n\right)  \neq0,$
\begin{align*}
&  \limsup_{R\rightarrow+\infty}\left\{  R^{-\frac{d-1}{2}}\left\Vert
\mathcal{D}\left(  R\Omega-x\right)  \right\Vert _{L^{p}\left(  \mathbb{T}%
^{d}\right)  }\right\}  \geqslant\limsup_{R\rightarrow+\infty}\left\{
R^{-\frac{d-1}{2}}\left\vert R^{d}\widehat{\chi_{\Omega}}(Rn)\right\vert
\right\} \\
&  \geqslant\frac{1}{\pi}\left\vert n\right\vert ^{-\frac{d+1}{2}}K^{-\frac
{1}{2}}\left(  \sigma\left(  n\right)  \right)  \limsup_{R\rightarrow+\infty
}\left\{  \left\vert \sin\left(  2\pi Rn\cdot\sigma\left(  n\right)  \right)
\right\vert \right\}  >0.
\end{align*}

The last part of the proof relies on the ideas of Parnovski and Sobolev in
\cite{PS}. We need a variant of Dirichlet's theorem on simultaneous
diophantine approximation (see \cite{PS}). Let $\alpha_{1},\ldots,\alpha_{m}$
be real numbers, then for every positive integer $j$ there exist integers
$s_{1},\ldots,s_{m}$ and $r$ such that%
\[
j\leqslant r\leqslant j^{m+1}\quad\text{and}\quad\left\vert r\alpha_{k}%
-s_{k}\right\vert <j^{-1}\text{ for every }k=1,\ldots,m.
\]
Let $\beta=\frac{2p}{2d-pd+p}$, and
\[
\left\{  \alpha_{k}\right\}  _{k=1}^{m}=\left\{  n\cdot\sigma\left(  n\right)
\right\}  _{\left\vert n\right\vert \leqslant j^{\beta}}.
\]
Then there exist integers $s_{n}$ and $R_{j}$ such that%
\[
j\leqslant R_{j}\leqslant j^{cj^{\beta d}+1}\leqslant j^{Cj^{\beta d}%
}\ \ \ \text{and}\ \ \ \text{\ }\left\vert R_{j}n\cdot\sigma\left(  n\right)
-s_{n}\right\vert \leqslant j^{-1}.
\]
It follows that%
\[
\left\vert \sin\left(  2\pi R_{j}n\cdot\sigma\left(  n\right)  \right)
\right\vert =\left\vert \sin\left(  2\pi\left(  R_{j}n\cdot\sigma\left(
n\right)  -s_{n}\right)  \right)  \right\vert \leqslant2\pi j^{-1}.
\]
By the Hausdorff-Young inequality in Theorem \ref{Thm Estimate}, for
$p\geqslant2$ and $1/p+1/q=1$, we have%
\[
\left\{  \int_{\mathbb{T}^{d}}\left\vert \mathcal{D}\left(  R_{j}%
\Omega-x\right)  \right\vert ^{p}dx\right\}  ^{1/p}\leqslant\left\{
\sum_{0\neq n\in\mathbb{Z}^{d}}\left\vert R^{d}\widehat{\chi_{\Omega}}\left(
Rn\right)  \right\vert ^{q}\right\}  ^{1/q}.
\]
If $2\leqslant p<\frac{2d}{d-1}$ then the above estimates of $\widehat{\chi
_{\Omega}}\left(  R_{j}n\right)  $ yield%
\begin{eqnarray*}
&&  \sum_{n\neq0}\left\vert R_{j}^{d}\widehat{\chi_{\Omega}}\left(
R_{j}n\right)  \right\vert ^{q} \\
& \leqslant & CR_{j}^{\left(  d-1\right)  q/2}\sum_{n\neq0}\left\vert
n\right\vert ^{-\left(  d+1\right)  q/2}\left\vert \sin\left(  2\pi
R_{j}n\cdot\sigma\left(  n\right)  \right)  \right\vert ^{q}  +\sum_{n\neq0}O\left(  R_{j}^{\left(  d-3\right)  q/2}\left\vert
n\right\vert ^{-\left(  d+3\right)  q/2}\right) \\
&  \leqslant & CR_{j}^{\left(  d-1\right)  q/2}\left(  j^{-q}\sum_{0<\left\vert
n\right\vert \leqslant j^{\beta}}\left\vert n\right\vert ^{-\left(
d+1\right)  q/2}+\sum_{\left\vert n\right\vert >j^{\beta}}\left\vert
n\right\vert ^{-\left(  d+1\right)  q/2}\right)  +O\left(  R_{j}^{\left(
d-3\right)  q/2}\right) \\
&  \leqslant & CR_{j}^{\left(  d-1\right)  q/2}\left(  j^{-q}+j^{\beta\left(
d-\left(  d+1\right)  q/2\right)  }\right)  +O\left(  R_{j}^{\left(
d-3\right)  q/2}\right)  .
\end{eqnarray*}
Since $\beta=\frac{2q}{q\left(  d+1\right)  -2d}$ and $R_{j}\geqslant j$, we
obtain%
\[
\left\{  \sum_{m\neq0}\left\vert R_{j}^{d}\widehat{\chi_{\Omega}}\left(
R_{j}n\right)  \right\vert ^{q}\right\}  ^{1/q}\leqslant cR_{j}^{\left(
d-1\right)  /2}j^{-1}+O\left(  R_{j}^{\left(  d-3\right)  /2}\right)
\leqslant CR_{j}^{\left(  d-1\right)  /2}j^{-1}.
\]
Finally, letting $j\rightarrow+\infty$ we obtain%
\begin{align*}
&  \liminf_{R\rightarrow+\infty}\left\{  R^{-\frac{d-1}{2}}\left\Vert
\mathcal{D}\left(  R\Omega-x\right)  \right\Vert _{L^{p}\left(  \mathbb{T}%
^{d}\right)  }\right\} 
 \leqslant\liminf_{R\rightarrow+\infty}\left\{  R^{-\frac{d-1}{2}}\left\{
\sum_{n\neq0}\left\vert R^{d}\widehat{\chi_{\Omega}}\left(  Rn\right)
\right\vert ^{q}\right\}  ^{1/q}\right\}  =0.
\end{align*}
More precisely, if $\varphi\left(  t\right)  =t^{\beta d}\log\left(  t\right)
$ then one can prove that, for large $s$,%
\[
\varphi^{-1}\left(  s\right)  \approx\left(  s\frac{\beta d}{\log\left(
s\right)  }\right)  ^{1/\beta d}.
\]
This implies that if $R_{j}\leqslant j^{Cj^{\beta d}}$ then%
\[
j^{-1}\leqslant C\left(  \frac{\log\left(  R_{j}\right)  }{\log\left(
\log\left(  R_{j}\right)  \right)  }\right)  ^{-1/\beta d}.%
\]
Therefore%
\[
\left\{  \sum_{m\neq0}\left\vert R_{j}^{d}\widehat{\chi_{\Omega}}\left(
R_{j}n\right)  \right\vert ^{q}\right\}  ^{1/q}\leqslant C\,R_{j}^{\left(
d-1\right)  /2}\left(  \frac{\log\left(  R_{j}\right)  }{\log\left(
\log\left(  R_{j}\right)  \right)  }\right)  ^{-1/\beta d}.
\]
\end{proof}

\bigskip

As we said, $\left\vert \widehat{\chi_{\Omega}}(\xi)\right\vert
\leqslant C\left\vert \xi\right\vert ^{-\left(  d+1\right)  /2}$ whenever
$\Omega$ has smooth boundary with positive Gaussian curvature. However, for
domains in the plane this smoothness assumption can be relaxed. Consider a
convex body $\Omega$ which can roll unimpeded inside a disc $\Delta$. This
means that for any point $x$ on the boundary $\partial\Delta$ there is a
translated copy of $\Omega$ contained in $\Delta$ that touches $\partial
\Delta$ in $x$.

\begin{figure}[h!]
\centering
\begin{tikzpicture}[line cap=round,line join=round,>=stealth,scale=0.35]
\filldraw [fill=green!25,fill opacity=0.5] plot[shift={(11.5358983849,-1.)},line width=1.pt,domain=-0.523598775598:0.523598775598,variable=\t]({1.*4.*cos(\t r)+0.*4.*sin(\t r)},{0.*4.*cos(\t r)+1.*4.*sin(\t r)}) --
plot[shift={(15.,-3.)},line width=1.pt,domain=1.57079632679:2.61799387799,variable=\t]({1.*4.*cos(\t r)+0.*4.*sin(\t r)},{0.*4.*cos(\t r)+1.*4.*sin(\t r)}) --
plot[shift={(15.,1.)},line width=1.pt,domain=3.66519142919:4.71238898038,variable=\t]({1.*4.*cos(\t r)+0.*4.*sin(\t r)},{0.*4.*cos(\t r)+1.*4.*sin(\t r)});
\draw [line width=1.pt] (11.,-3.) circle (5.65685424949cm);
\draw [line width=0.5pt,dash pattern=on 3pt off 3pt] (11.2824898206,2.64979641237)-- (16.6519159808,-2.7636829555);
\draw [->,line width=1.pt] (5.,0.)-- (9.,4.);
\draw [<->,line width=0.2pt,dash pattern=on 1pt off 1pt] (15,1) -- (13.96,-0.04);
\draw (12.662833788,-0.728652444091) node[anchor=north west] {$\Omega$};
\draw (8,-3) node[anchor=north west] {$\Delta$};
\draw (5.97500419166,3.09871991528) node[anchor=north west] {$\vartheta$};
\draw (14.1,.8) node[anchor=north west] {$\scriptsize{\delta}$};
\end{tikzpicture}
\end{figure}

\begin{theorem}
\textit{If a planar convex set }$\Omega$\textit{\ can roll unimpeded inside a
disc, then}%
\[
\left\vert \widehat{\chi_{\Omega}}(\xi)\right\vert \leqslant C\left\vert
\xi\right\vert ^{-3/2}.
\]

\end{theorem}

\begin{proof}
In \cite{P} (see also \cite{B-R-T, T}) it is proved that if
\[
\lambda\left(  \delta,\vartheta,\Omega\right)  =\left\vert \left\{  x\in
\Omega:\delta+x\cdot\vartheta=\sup_{y\in\Omega}\left\{  y\cdot\vartheta
\right\}  \right\}  \right\vert 
\]
(this is the length of the chord perpendicular to the outward direction
$\vartheta$ and at a small distance $\delta$ from the boundary $\partial
\Omega$), then
\[
\left\vert \widehat{\chi_{\Omega}}(\rho\vartheta)\right\vert \leqslant
\frac{\mathrm{diameter}(\Omega)}{2\rho}\left(  \lambda\left(  (2\rho)^{-1},\vartheta,\Omega\right)
+\lambda\left(  (2\rho)^{-1},-\vartheta,\Omega\right)  \right)  .
\]
If $\Omega$ can roll unimpeded inside a disc $\Delta$, then $\lambda\left(
\delta,\vartheta,\Omega\right)  \leqslant\lambda\left(  \delta,\vartheta
,\Delta\right)  $. This implies that the Fourier transform of $\Omega$ is
dominated by the chords of a disc, and therefore $\left\vert \widehat{\chi
_{\Omega}}(\xi)\right\vert \leqslant C\left\vert \xi\right\vert ^{-3/2}$.
\end{proof}

A curve can roll unimpeded inside another curve if and only if the largest
radius of curvature of the first is smaller than the smallest radius of
curvature of the second. No smoothness of these curves is required, the
rolling curve may also have corners. See \cite[Chapter 17]{B-F} and the
references therein.

In particular, the above results give an alternative proof of the result in
\cite{Hu}.

\begin{corollary}
\textit{Let }$\Omega$ be \textit{a planar convex set that can roll unimpeded
inside a disc. For any }$R>2$ we have%
\[
\left\{  \int_{\mathbb{T}^{2}}\left\vert \mathcal{D}\left(  R\Omega-x\right)
\right\vert ^{4}dx\right\}  ^{1/4}\leqslant CR^{1/2}\log^{1/4}\left(
R\right)  .
\]

\end{corollary}

\end{document}